\documentclass{amsart}
\usepackage{setspace}
\usepackage{amsmath,amscd, amsthm,verbatim}
\usepackage{amsfonts}
\usepackage[english]{babel}
\usepackage{amssymb}
\usepackage{enumerate}
\usepackage{mathrsfs}
\usepackage{color}
\usepackage[pagebackref=false,colorlinks=false, pdfstartview=FitV, linkcolor=blue,citecolor=red, urlcolor=blue]{hyperref}

\newcommand{\R}{\mathbb{R}} 
 
\newcommand{\C}{\mathbb{C}} 
\newcommand{\Z}{\mathbb{Z}} 
\newcommand{\Q}{\mathbb{Q}}

\newcommand{\fa}{\mathfrak{a}}

\def\PSL{\ensuremath {\mathrm{PSL}}}
\def\SL{\ensuremath {\mathrm{SL}}}

\newcommand{\tab}{\hspace*{15pt}}

\newcommand{\xqedhere}[2]{%
  \rlap{\hbox to#1{\hfil\llap{\ensuremath{#2}}}}}

\usepackage{tikz-cd}
\usepackage{enumerate}
\usepackage{mathrsfs}
\usepackage[utf8]{inputenc}
\usepackage[T1]{fontenc}
\usepackage{xypic}
\input xy

\newtheorem{thm}{Theorem}

\newtheorem{lem}[thm]{Lemma}
\newtheorem{prop}[thm]{Proposition}

\theoremstyle{remark}
\newtheorem{rek}[thm]{Remark}

\numberwithin{equation}{section}

\title[The equidistribution of Elliptic Dedekind sums]{The equidistribution of Elliptic Dedekind sums and generalized Selberg-Kloosterman sums}
\author{Kim Klinger-Logan}
\address{Kansas State University, 1228 N Martin Luther King Jr Dr, 138 Cardwell Hall, Manhattan, KS 66506, USA}
\email{kklingerlogan@ksu.edu}
\author{Tian An Wong}
\address{University of Michigan-Dearborn, 4901 Evergreen Rd, 2002 CASL Building, Dearborn, MI 48128, USA}
\email{tiananw@umich.edu}
\date{\today} 

         
\begin{document}
\begin{abstract}
We show that the values of elliptic Dedekind sums, after normalization, are equidistributed mod 1. The key ingredient is a non-trivial bound on generalized Selberg-Kloosterman sums for discrete subgroups of $\PSL_2(\mathbb C)$ using Poincar\'e series. 
\end{abstract}
\maketitle

\section{Introduction}


\subsection{The distribution of elliptic Dedekind sums}

Let $c,d$ be relatively prime integers. The classical Dedekind sum is given by
\[
s(c,d) = \sum_{n=1}^c((n/c))((nd/c))
\]
where $((x)) := \left\{ \begin{array}{ll} \{x\}-\frac12 &\text{for }x\in\mathbb{R}-\mathbb{Z}, \\0 & \text{for }x\in\mathbb{Z}\end{array}\right.$ and $\{x\}$ is the fractional part of $x$. Following conjectures of Grosswald and Rademacher \cite{RG}, Hickerson showed that the image of $s(c,d)$ is dense in $\mathbb{R}$ \cite{H}, and Vardi later showed that classical Dedekind sums are furthermore equidistributed mod 1 \cite{V}. Viewing $s(c,d)$ as a function on $\SL_2(\Z)$, a generalization of this result to non-cocompact lattices in $\SL_2(\R)$ was also recently proven by Burrin \cite{B}. 

Elliptic Dedekind sums are analogues of classical Dedekind sums for imaginary quadratic fields. The goal of this paper is to study the value distribution of elliptic Dedekind sums.  To further describe our results and motivation, we introduce some basic notation. Let $L=\omega_1\Z+\omega_2\Z$ be a lattice in $\mathbb{C}$ with $\text{Im}(\omega_1/\omega_2)>0$ and let $\mathcal{O}_L:=\{m\in\C~|~mL\subset L\}$ be its ring of multipliers.
For $x\in\mathbb{C}$ and $k\in\mathbb{Z}_{\ge 0}$, define the Eisenstein-Kronecker series
\[
E_k(x):=\sum_{\substack{w\in L\\w+x\neq 0}} (w+x)^{-k}|w+x|^{-s}\Big|_{s=0},
\]
where the value at $s=0$ is understood in the sense of analytic continuation. 
Using these, Sczech \cite{Sczech} defines the elliptic Dedekind sum, for $c,d\in\mathcal{O}_L$ and $c\neq 0$,
\[
{D}(c,d) := \frac{1}{c}\sum_{r\in L/cL} E_1\left(\frac{rd}{c} \right) E_1\left(\frac{r}{c} \right),
\]
and an additive  homomorphism $\Phi: \SL_2(\mathcal{O}_L)\to \C$ given by
\begin{equation}
\label{eq:phi}
\Phi\left(\begin{matrix}a & b \\ c & d\end{matrix}\right):= \left\{ \begin{array}{lc}
     E_2(0)I\left(\frac{a+d}{c}\right)- {D}(c,d)& c\neq 0 \\
    E_2(0)I\left(\frac{b}{d}\right) & c=0 
    \end{array}\right. 
\end{equation}
where $I(z):=z-\bar z = 2i\text{Im}(z)$ and $i = \sqrt{-1}$. 
Note that $E_2(0)$ vanishes if only if $\mathcal{O}_L$ is equal to $\mathbb{Z}$, $\mathbb{Z}[i]$, or $\mathbb{Z}[\rho]$ where $\rho=(-1+\sqrt{-3})/2$.
  
Let $K$ be an imaginary quadratic number field with discriminant $d_K$. We take $L$ to be the ring of integers $\mathcal O_K$, so $\mathcal O_L = \mathcal O_K$. Define also the normalizations, for $c\neq 0$,
\[
\widetilde D(c,d)=D(c,d)/(2E_2(0)\sqrt{-1})\text{ and }\widetilde\Phi=\Phi/(2E_2(0)\sqrt{-1}).
\]
In this paper weassume that $K$ is a fixed imaginary quadratic field different from $\mathbb{Q}(\sqrt{-1})$ and $\mathbb{Q}(\sqrt{-3})$ (in which case $D$ is identically zero). In \cite{REU}, Berkopec, Branch, Heikkinen, Nunn, and Wong showed that the values of $\widetilde{D}$ are real and dense in $\mathbb{R}$, extending work of Ito \cite{Ito} from Euclidean $K$ to the general case. Following Vardi's work in the classical case, it is then natural to ask how the values of $\widetilde{D}$ are distributed. 

In this paper, we show that the values of the normalized elliptic Dedekind sum are equidistributed mod $1$. 

\begin{thm}
\label{main}
For all positive $r\in\mathbb R$, 
 the sequence 
\[
\{r\widetilde{D}(c,d)\,:\, c,d\in \mathcal{O}_K, 0<|d|<|c|, (c,d)=1\} 
\]
is equidistributed on $[0,1)$ as $|c|\to\infty$.
\end{thm}

\noindent  To prove Theorem \ref{main}, by the Weyl criterion for equidistribution, it suffices to obtain a nontrivial bound on the following sum for any positive integer $n$ and real $X>0$,
\[
\sum_{0 < |c| < X }\sum_{\substack{d \bmod c\\ (c,d)=1}}e(rn\widetilde{D}(c,d)), 
\]
where  $r\in\mathbb{R}$ and $e(x) = e^{2\pi i x}$ (the trivial bound being $O(X^4)$). If we let $\Gamma = \PSL_2(\mathcal O_K)$ and $\Gamma_\infty$ be the stabilizer of the cusp at infinity, then the indexing set corresponds to a complete set of representatives of the quotient $\Gamma_\infty\backslash \Gamma$.  Using the homomorphism \eqref{eq:phi}, the inner sum transforms into a generalized Selberg-Kloosterman sum associated to $\Gamma$ and $\chi$ at the cusp $\infty$,
\begin{align}\label{def:S1}
\sum_{\substack{d \bmod c\\ (c,d)=1}} e(r \widetilde D(c,d))  =S_\infty(\lfloor -r\rfloor,\lfloor-r\rfloor,c,{\chi_\alpha})
\end{align}
for $c\neq 0$, where we set $\alpha = -r + \lfloor-r\rfloor$ with $\lfloor x\rfloor $ the greatest integer less than or equal to $x$ and 
\begin{equation}
  \label{chi}
{\chi_\alpha(M)}=e(\alpha\widetilde \Phi(M)),\qquad \alpha\in\mathbb{R}.
\end{equation}
A more general definition is given in (\ref{def:kloos}). Note that $\widetilde\Phi$ is integral by \cite[Satz 4]{Sczech} and moreover real-valued. The function $S_\infty$ in defined in \eqref{def:kloos} and  can also be viewed as a Kloosterman sum twisted by a  unitary character related to the `phase' of the elliptic Dedekind sum.

 In the classical case, the required bounds for Selberg-Kloosterman sums were proved by Goldfeld and Sarnak \cite{GS}. The equidistribution of classical Dedekind sums was proved by Vardi using a similar relation between Dedekind sums and Selberg-Kloosterman sums on $\SL_2(\Z)$ \cite[Theorem 1.4]{V}, which take the form
 \[
 e\left(\frac{r}{4}\right)\sum_{\substack{0 < d < c\\(c,d)=1}}e(rs(c,d)) = S\left(\Big\lfloor\frac{-r}{12}\Big\rfloor,\Big\lfloor\frac{-r}{12}\Big\rfloor,c,\chi_r,\mathrm{SL}_2(\Z)\right),
 \]
 for $c\neq 0$ and where $S$ is defined as on \cite[p.191]{V}.
 We will use a similar strategy to prove Theorem \ref{main} in a more general setting.

\subsection{Sums of generalized Selberg-Kloosterman sums}
Let $\Gamma$ be a cofinite and non-cocompact discrete subgroup of $\PSL_2(\mathbb C)$ with an associated cusp $\zeta\in \mathbb P^1(\C)$ and $A\in \PSL_2(\mathbb{C})$ so that $A\zeta=\infty$. 
Let $\Gamma_\zeta$ be the stabilizer of $\zeta$ in $\Gamma$, and $\Gamma'_\zeta$ its maximal unipotent subgroup. Let $\mathcal R$ be a complete set of representatives of double cosets in the quotient
 \begin{equation}
 \label{Rset}
\gamma = \begin{pmatrix} * &* \\ c & d\end{pmatrix} \in 
A \Gamma'_\zeta A^{-1}\backslash  A\Gamma A^{-1}/ A \Gamma'_\zeta A^{-1}, \quad c\neq 0,
\end{equation}
which we also write as $(c,d)$ for short. Given $\Gamma'_\zeta$, there is an associated complex lattice $\Lambda$ described in \eqref{eq:Lambdadef}, with dual lattice $\Lambda'$. Note that $\Lambda'$ is non-standard and is defined with respect to the inner product in \eqref{eq:innerproduct}. For this reason, the Selberg-Kloosterman sums we define are also non-standard.

Define a bilinear form on $\C^2$ as 
\begin{equation}\label{eq:innerproduct}
\langle u,v \rangle  = \frac{u\bar v - \bar u v}{2i} = \text{Im}(u\bar v).
\end{equation}
 We note that this choice of bilinear form generally differs from convention, such as in \cite{EGM}; this difference is due to the fact that \eqref{eq:phi} involves the function $I(z)$. 
For any fixed $\alpha\in\C$, let $\chi_\alpha:\Gamma\to S^1$ be a homomorphism such that 
 for each lattice $\Lambda$ associated to a cusp $\zeta$ of $\Gamma$,
\begin{equation}
\label{eq:chiprop}
\chi_\alpha:\begin{bmatrix} 1 & \omega \\ 0 & 1 \end{bmatrix} \mapsto e\left(-\text{Im}(\alpha\bar\omega)\right) = e(-\langle \alpha,\omega\rangle), \qquad \omega\in\Lambda.
\end{equation}
 This includes the special case of $\Gamma = \PSL_2(\mathcal O_K)$ and $\alpha\in \mathbb R$ with $\chi_\alpha$ as in \eqref{chi}, which follows from the fact that
for any $a\in\mathcal O_K$, we have $\widetilde\Phi(\begin{smallmatrix}1 & a \\ 0 & 1\end{smallmatrix}) =  \text{Im}(a)$ and also $\overline{\Phi(M)}=-\Phi(M)$ by \cite{Sczech}.

For any $m,n\in\Lambda',c\in\Lambda$, and $\alpha\in\mathbb C$, we define the generalized Selberg-Kloosterman sum on $\Gamma$ associated to $\zeta$ and $\chi_\alpha$ to be
\begin{equation}\label{def:kloos}
S_{\zeta}(m,n,c,\chi_\alpha):= \sum_{(c,d) \in \mathcal R}
\overline{\chi_\alpha (\gamma)}\, \widetilde e\left(\frac{\overline{(m-\alpha)}a+\overline{(n-\alpha)}d}{c}\right),
\end{equation}
where $\widetilde e(z) = e^{2\pi i \text{Im}(z)}$. 
Note that in the original Selberg-Kloosterman sums over $\Q$, the character $\chi$ is obtained from the multiplier system associated to a given factor of automorphy. For this reason, the estimate given in Theorem \ref{thm:GSthm2}  is of independent interest.

The proof of Theorem \ref{main} follows from bounds on both the generalized Selberg-Kloosterman sums and the Kloosterman zeta function given by 
 \begin{align}
 \label{zeta}
 Z_\zeta(m,n,s, \chi_\alpha ) := \sum_{c\neq0}\frac{S_\zeta(m,n,c, \chi_\alpha)}{|c|^{2s}}.
 \end{align} 
The main ingredients in the proof of Theorem \ref{main} are the associated Eisenstein and Poincar\'e series.  These Poincar\'e series give bounds for our generalized Selberg-Kloosterman sums and hence generalize a classical result of Goldfeld and Sarnak  \cite{GS} to the imaginary quadratic case. For simplicity, we state the theorem for the case $\zeta =\infty$.

\begin{thm}
\label{thm:GSthm2} 
Let $m,n\in\Lambda'$, $c\in \Lambda$ and $\alpha\in\mathbb{C}$. Let $\chi_\alpha$  be a unitary character of $\Gamma$ satisfying \eqref{eq:chiprop}. Then for any  $\epsilon>0$, we have
 $$
 \sum_{|c|\leq x}\frac{S_\infty(m,n,c,\chi_\alpha)}{|c|^2}=\sum_{j=1}^\ell \tau_j x^{\beta_j}+O(x^{1+\epsilon}),
 $$
 for certain constants $\tau_j$ and real $ 0<\beta_j<2$. The implied constant depends only on $\Gamma, \chi_\alpha, k, m $ and $n$.
\end{thm}

\noindent Theorem \ref{thm:GSthm2} can be adapted to hold for other $\Gamma$ and $\zeta$ by applying the argument in Section \ref{sec:bounds} to the Poincar\'e series defined in \eqref{eq:emdefgeneral}.

 In Section \ref{proofothm1}, together with a routine application of the analytic theory of Eisenstein series on $\Gamma$, we deduce the main result by specializing to $\Gamma = \PSL_2(\mathcal O_K)$ and $K\neq \mathbb{Q}(\sqrt{-1}), \mathbb{Q}(\sqrt{-3})$. 
The general nature of these results draws, in part, on the work of Elstrodt, Grunewald, and Mennicke \cite{EGM}. In particular, similar to the work of Burrin \cite{B} in the classical case, our methods should apply to future generalizations of elliptic Dedekind sums to more general subgroups of $\PSL_2(\C)$.



\section{Definitions}\label{background}
We first recall some basic definitions and fix the notation that we require, referring to \cite{EGM} and \cite{Sarnak} for further details.

\subsection{Upper Half-Space}
The upper half-space in Euclidean 3-space is a model for 3-dimensional hyperbolic space
$$
\mathbb{H}^3=\{ (z,r)~|~ z\in \mathbb{C}, r>0\}= \{ (x,y,r)~|~ x,y\in \mathbb{R}, r>0\}.
$$
We can view $\mathbb{H}^3$ as a subset of Hamilton's quaternions, and a point $P\in \mathbb{H}^3$ is written 
\[
P=(z,r)= (x,y,r)= z+rj,
\]
where $z=x+iy$ and $j=(0,0,1)$.
 The action of $\SL_2(\mathbb{C})$ on $P\in \mathbb{H}^3$ is given by the isometry
 $$
 P\mapsto MP:= M(P) = (aP+b)(cP+d)^{-1},
 $$
 where the inverse is taken in the skew field of the quaternions. More explicitly, we write
$M(z+rj)=z_M+r_Mj$ where 
$$
z_M =\frac{(az+b)(\bar c\bar z+\bar d) +a\bar c r^2}{|cz+d|^2+|c|^2r^2}\ \ \ \text{ and } \ \ \ r_M = \frac{r}{|cz+d|^2+|c|^2r^2}.
$$
We are interested in discrete subgroups $\Gamma$ of $\PSL_2(\mathbb C)$ and functions on the quotient $\Gamma\backslash \mathbb H^3$.  Define $D_\Gamma$ to be a fundamental domain of $\Gamma$. Define also the Laplace-Beltrami operator
\[
\Delta =  r^2\left(\frac{\partial^2}{\partial^2 x}+\frac{\partial^2}{\partial^2 y}+\frac{\partial^2}{\partial^2 r} \right) - r\frac{\partial}{\partial r},
\]
and the associated volume form $ dv = dx\,dy\,dr/r^3$ and length form \\$ds^2 = (dx^2+dy^2+dr^2)/r^2$.

We write $\Gamma_\infty$ for the stabilizer of the cusp at infinity in $\Gamma$. In the case that $\Gamma = \PSL_2(\mathcal O_K)$, we have that
\[
\Gamma_\infty = \left\{\begin{pmatrix}1 & a \\ 0 & 1\end{pmatrix} : a \in\mathcal O_K\right\},
\]
and a coset representation of $\Gamma_\infty\backslash \Gamma$ is given by
\[
\left\{\begin{pmatrix}* & * \\ c & d\end{pmatrix}\in\Gamma : (c,d)=1\right\}.
\]
More generally, given a cusp $\zeta$ of a general $\Gamma$ such that $A\zeta=\infty$, the set
\begin{equation}\label{eq:Lambdadef}
\Lambda := \left\{\omega\in \mathbb{C}: 
\begin{pmatrix}
     1 & \omega\\
     0 & 1
\end{pmatrix} 
\in A\Gamma A^{-1}\right\}
\end{equation}
forms a lattice in $\C$. We shall denote by $\Gamma_\zeta$ the stabilizer of $\zeta$ in $\Gamma$ and $\Gamma'_\zeta$ the maximal unipotent subgroup of $\Gamma_\zeta$. 
We have then that 
\[
A\Gamma'_\zeta A^{-1}= \left\{
\begin{pmatrix}
     1 & \omega\\
     0 & 1
\end{pmatrix} 
: \omega\in\Lambda\right\}.
\]
Denote by $|\Lambda|$ the area of a fundamental parallelogram associated to the lattice $\Lambda$, and $\mathrm{vol}(\Gamma)$ the covolume of $\Gamma$ obtained by integrating $ds$ over a fundamental domain of $\Gamma$. For example, if $\Gamma = \PSL_2(\mathcal{O}_K),$ a well-known result of Humbert shows that 
\[
\mathrm{vol}(\Gamma) = {|d_K|^{3/2}}\zeta_K(2)/(4\pi^2),
\]
where $d_K$ is the discriminant of $K$ and $\zeta_K(s)$ is the Dedekind zeta function of $K$ \cite[Chapter 7, Theorem 1.1]{EGM}.

\subsection{Eisenstein series}
\label{eis}
The Eisenstein series on $\Gamma$ associated to the cusp $\zeta$ is defined by
\[
E_{A}(P,s):=\sum_{M\in\Gamma'_{\zeta}\setminus\Gamma} (r_{AM})^{s},
\]
which converges absolutely for Re$(s)>2$. The map 
$
z+rj\mapsto r^{s}
$
satisfies the differential equation $\Delta r^{s} = s(s-2)r^{s},$ so that $E_{A}(P,s)$  is a $\Gamma$-invariant eigenfunction of $\Delta$ with
\[
\Delta E_A(P,s) = s(s-2)E_A(P,s),
\]
but is not square-integrable on $\Gamma\backslash\mathbb H^3$. By Theorem 6.1.2 of \cite{EGM}, $E_{A}(P,s)$ has meromorphic continuation to the whole $s$-plane. Furthermore, Theorem 6.1.11 of \cite{EGM} shows $E_{A}(P,s)$ is holomorphic for $\mathrm{Re}(s)>0$, except for a finite number of simple poles on the real line $(0,2]$, and the simple pole of $E_{A}(P,s)$ at $s=2$ has residue equal to $|\Lambda|/\mathrm{vol}(\Gamma)$.  Note also that the series
\[
E_{A}^*(P,s)=\sum_{M\in(A\Gamma'_{\zeta}A^{-1})\setminus A\Gamma A^{-1}} (r_{M})^{s},
\]
agrees with $E_A(P,s)$ term by term by the change of variable $M \mapsto AMA^{-1}$ whenever the series converges absolutely.

The analytic properties of $E_A(P,s)$ are controlled by the behaviour of the zeroth coefficient of its Fourier expansion at the cusp. The Fourier expansion is given by
\[
\sum_{\omega'\in\Lambda'}a_{\omega'}(r,s)q^{\omega'},
\]
where $q^{\omega'} =e(\omega' \cdot z),$ with $e(x) = e^{2\pi i x}$ and the operation $\cdot $ denotes the Euclidean scalar product in $\C$, and $\Lambda'$ is the dual lattice of $\Lambda$ with respect to the latter inner product. The Fourier coefficients can be expressed as
\begin{equation}
\label{a0}
    a_{0}(r,s) = [\Gamma_{\zeta}:\Gamma_{\zeta}'] |d_0|^{-2s} r^{s} + \frac{\pi}{(s-1)|\Lambda|}\phi_0(s)r^{2-s}
\end{equation}
and
\[
    a_{\omega'}(r,s)= \frac{2\pi^{s}}{|\Lambda|\Gamma(s)} |\omega'|^{s-1} \phi_{\omega'}(s)r K_{s-1}(2\pi|\omega'|r), \qquad \omega'\neq0,
\]
where $d_0$ is a constant given in \cite[Theorem 2.1]{EGM}, $K_s$ denotes the modified Bessel function of the second kind, and $\phi_{\omega'}(s)$ is the Dirichlet series
\[
\phi_{\omega'}(s) = \sum_{(c,d)\in \mathcal R}\frac{e(\omega'\cdot\frac{d}{c})}{|c|^{2s}},
\]
with the summation running over $\mathcal R$ as defined in \eqref{Rset}. 

\begin{rek}
For example, if $\Gamma = \PSL_2(\mathbb{Z}[i])$, then $\phi_{0}(s) = \zeta_K(s-1)/\zeta_K(s)$ and for $\Gamma = \PSL_2(\mathcal O_K)$, it can be written as
\[
\sum^h_{j=1}\frac{L_j(s-1)}{L_j(s)},
\]
where $h$ is the class number of $K$ and
$
L_j(s)=\sum_{\mathfrak a }\psi_j(\mathfrak a)N(\mathfrak a)^{-s},
$
the sum running over nonzero ideals of $\mathcal O_K$ and $\psi_j$ an ideal class character \cite[p.210]{Sarnak}. For general $\Gamma$, the functional equation for $\phi_0$ follows from the functional equation of the associated Eisenstein series (see \cite[p.\,232-233]{EGM}).
\end{rek}

\subsection{Poincar\'e series}\label{sec:PS} 
For $s\in \mathbb{C}$ and $0\neq m\in  \Lambda'$, the dual lattice of $\Lambda$, we define the Poincar\'{e} series on $\Gamma$ associated to $\chi_\alpha$ as defined in \ref{eq:chiprop} and $\zeta$ as
\begin{align}\label{eq:emdefgeneral}
U_{A,m}(P,s, \chi_\alpha,\Gamma,\zeta) := \sum_{M\in \Gamma '_\zeta\backslash \Gamma }\overline{\chi_\alpha(AM)}\,r_{AM}^s e(i|m-\alpha|r_{AM}-\langle m - \alpha,z_{AM}\rangle)
\end{align}
where the inner product is as defined in \eqref{eq:innerproduct}.
The series converges absolutely for $s\in\mathbb{C}$ and $\text{Re}(s)>2$ following the same argument as \cite[\S3]{Sarnak}. 
For the purpose of clarity, we will set $\zeta=\infty$ in the future proofs and calculations 
 and we more simply define
 \begin{align}
U_m(P,s, \chi_\alpha) := \sum_{M\in \Gamma '_\infty\backslash \Gamma }\overline{\chi_\alpha(M)}\,r_M^s e(i|m-\alpha|r_M-\langle m - \alpha,z_M\rangle).
\end{align}
These computations can be applied to $U_{A,m}(P,s, \chi_\alpha,\Gamma,\zeta)$ with proper scaling. 

\begin{lem}
The function $U_m$ is $\Gamma_\infty'$-invariant, and
\begin{equation}\label{eq:autU}
U_m(MP,s, \chi_\alpha) =\chi_{\alpha}(M)\,U_m(P,s, \chi_\alpha).
\end{equation}
\end{lem}

\begin{proof}
For any $M'\in \Gamma_\infty'$, we have that $M'$ acts trivially on $r_M$ and by translation on $z_M$. We then have, up to conjugaton, 
\begin{align*}
&\overline{\chi_\alpha(M'M)} r_{M'M}^s e(i|m-\alpha|r_{M'M}-\langle m - \alpha,z_{M'M} \rangle),\qquad M' = \begin{bmatrix} 1 & \omega \\ 0 & 1 \end{bmatrix}\\
&=e(\langle \alpha,\omega\rangle)\overline{\chi_\alpha(M)} r_M^s e(i|m-\alpha|r_{M}-\langle m - \alpha,z_{M}+ \omega \rangle)\\
&= \overline{\chi_\alpha(M)} r_M^s e(i|m-\alpha|r_{M}-\langle m-\alpha,z_{M} \rangle),
\end{align*}
where we note that $\langle m, \omega\rangle\in \Z$ so that $e(\langle m,\omega\rangle)=1$. 

Secondly, by direct computation we have
\begin{align*}\label{eq:autU}
U_m(M_1P,s, \chi_\alpha) 
  &=\overline{\chi_\alpha(M_1^{-1})} U_m(P,s, \chi_\alpha),
  \end{align*}
 and by the homomorphism property of $\chi_\alpha$, the identity \eqref{eq:autU} follows.
\end{proof}

Let $-\widetilde\Delta$ be the self-adjoint extension of the Laplace operator on $L^2(\Gamma\backslash\mathbb{H}^3)$. Let $\rho(-\widetilde\Delta)$ be the resolvent set of $-\widetilde\Delta$ and 
\[
\mathscr{R}_\lambda:= (-\widetilde\Delta - \lambda)^{-1},
\]
for $\lambda\in\rho(-\widetilde\Delta)$ the corresponding resolvent operator. Note that $\mathscr{R}_{s(2-s)}$ is holomorphic for $\text{Re}(s)>2$ and meromorphic for $\text{Re}(s)>1$ with at most finitely many poles of order $1$ in the interval $(1,2)$.

Similarly to \cite[p.355]{EGM} (see also \cite[(2.16)]{Sarnak}), we also have 
\begin{align}\label{eq:resID}U_m(P,s,\chi_\alpha)= 2\pi |m-\alpha|(1-2s)\mathscr{R}_{s(2-s)}U_m(P,s+1,\chi_\alpha)
\end{align}
since 
\begin{align*}\Delta U_m(&P,s,\chi_\alpha) +s(2-s)U_m(P,s,\chi_\alpha) =2\pi |m-\alpha|(1-2s)U_m(P,s+1,\chi_\alpha).\end{align*}
In particular, the region of convergence for $U_m(P,s,\chi_\alpha)$ extends meromorphically to $\text{Re}(s)>1$ with at most a finite number of poles in $(1,2)$.



\section{Bounds for generalized Selberg-Kloosterman sums}\label{sec:bounds}



In this section, we take $\Gamma$ to be a general cofinite but not cocompact discrete subgroup of $\PSL_2(\C)$, $\zeta = \infty$ a cusp of $\Gamma$, and $\chi_\alpha$ a homomorphism on $\Gamma$ as in Section \ref{sec:PS}. Let $D_\Gamma$ denote a fundamental domain for $\Gamma$ in $\mathbb{H}^3$ and $\mathcal{P}$ a fundamental polytope.

\subsection{Preliminary estimates}

We shall first prove an estimate for the Kloosterman zeta function in \eqref{zeta}.  We continue to assume $m,n$ are in the dual lattice of $\Lambda$.

{\thm\label{thm:GSthm1} The function $Z_{\infty}(m,n,s,\chi_\alpha) $ is meromorphic in $\text{Re}(s)>1$ with at most a finite number of simple poles in $(1,2)$ and satisfies the growth condition 
$$
|Z_{\infty}(m,n,s,\chi_\alpha)|=O\left(\frac{|n-\alpha|^2|m-\alpha|^2|s|}{(\sigma-1)^2}\right),
$$
for $s=\sigma+it$, $\sigma>1$ as $t\to \infty$. The implied constant depends only on $\Gamma, \chi_\alpha,m,n$.\\}

\noindent  For the proof of Theorem \ref{thm:GSthm1}, we  rely on two lemmas and Stirling's formula.  The first result we need is an upper bound for the $L^2$-norm of our Poincar{\'e} series. The proof and formulation of the following lemma is similar to that of \cite[Theorem 1]{GS}.

{\lem\label{lem1} Let $s=\sigma+it$ for $1<\sigma\leq 2$ and $|t|>1$,
$$
\int_{D_\Gamma} |U_m(P,s, \chi_\alpha)) |^2\, \frac{dx\,dy\,dr}{r^3}=O\left(\frac{|m-\alpha|^2}{(1-\sigma)^2} \right),
$$
where $D_\Gamma$ is a fundamental domain for $\Gamma\backslash \mathbb H^3$.}

\begin{proof}
For $2\leq \sigma\leq 3$, we have 
$
U_m(P,s,\chi_\alpha)=O(1)
$
uniformly in $\mathcal{P}$ since the series converges absolutely. From (\ref{eq:resID}) and  the property that $|\mathscr{R}_\lambda|\leq |\text{Im}(\lambda)|^{-1}$ by \cite[p.163]{EGM}, it follows that
$$
\left(\int_{D_\Gamma} |U_m(P,s, \chi_\alpha) |^2\, \frac{dx\,dy\,dr}{r^3}\right)^{1/2}=O(|m-\alpha|)\frac{|2s-1|}{2|t| |1-\sigma|},
$$
since $\text{Im}(s(2-s)) = 2t(1-\sigma)$.
\end{proof}


The next lemma computes the inner product of two Poincar\'e series.

{\lem\label{lem2} Let $\Gamma$ be a discrete subgroup of $\PSL_2(\mathcal{O}_K)$. For  $m,n\neq 0$, we have 
\begin{align*}
\int_{D_{\Gamma}}U_n(P,s, \chi_\alpha)\,& \overline{U_m(P,\bar s+2, \chi_\alpha)} \,  \frac{dx\,dy\,dr}{r^3}\\ & =\frac{\pi^{-3/2}4^{-s-1}\Gamma(2s)}{|n-\alpha|^{2}\Gamma(s)\Gamma(s+3/2)} Z_\infty(m,n,s,\chi_\alpha)+R(s),
\end{align*}
where $R(s)$ is analytic in $\text{Re}(s)>1$ and $|R(s)|= O\left(\frac{1}{\sigma-1} \right)$ in this region.}

 \begin{proof}
First note that from (\ref{eq:autU})  and unwinding 
 \begin{align*}
\int_{D_\Gamma}&U_m(P,s_1, \chi_\alpha)\,\overline{U_n (P,s_2, \chi_\alpha)} \, \frac{dx\,dy\,dr}{r^3}\\
  &=\int_{\Gamma\backslash\mathbb{H}^3} \sum_{M\in \Gamma_\infty'\backslash \Gamma }U_m(MP,s_1, \chi_\alpha)\chi_{ \alpha}(M)r_M^{\bar s_2} e(-i|n-\alpha|r+\langle n-\alpha,z\rangle) \, \frac{dx\,dy\,dr}{r^3}\\
&   = \int_0^\infty r^{\bar s_2}e^{2\pi |n-\alpha|r}\int_{\mathcal{P}} U_m(P,s_1, \chi_\alpha)e^{2\pi i\langle n-\alpha,z\rangle} \, dx\,dy\,\frac{dr}{r^3}.
\end{align*}
Examining the inner integral, 
\begin{align*}
\int_{\mathcal{P}}&U_m(P,s, \chi_\alpha)\, e^{2\pi i \langle n-\alpha, x+iy\rangle } \, dx\,dy\\
 & = \int_{\mathcal{P}}\sum_{M\in \Gamma '_\infty\backslash \Gamma }\overline{\chi_\alpha(M)}\,r_M^s e(i|m-\alpha|r_M-\langle m-\alpha,z_M\rangle )\, e^{2\pi i \langle n-\alpha, x+iy\rangle } \, dx\,dy,
 \end{align*}
we will break this sum into two terms, where $c=0$ and where $c\neq0$.  


For the contribution of the elements $M\in  \Gamma '_\infty\backslash \Gamma$ where $c=0$, we follow \cite{EGMpreprint}.
First notice that when $c=0$, 
 we have $|a|=|d|=1$. Thus $z_M= \varphi_a(z) +b\bar d$ where $\varphi_a:\mathbb{C} \to \mathbb{C}$ is an orthogonal linear map (as in \cite{EGMpreprint} p. 676). Denote $\varphi_a^* $ as the dual map of $\varphi_a$ with respect to $\langle \cdot,\cdot \rangle$. Then 
$$
\int_{\mathcal{P}} r^s e^{-2\pi |m-\alpha|r-2\pi i \langle  \varphi_a^*(m)-n\rangle, z} dz = r^se^{2\pi |m-\alpha|r}\mathrm{vol}(\mathcal{P})\delta_{\varphi_a^*(m), n},
$$
where $\delta$ is the Kronecker delta function. Thus 
\begin{align*}
 \sum_{\substack{M\in \Gamma _\infty'\backslash \Gamma\\c= 0} } &\overline{\chi_\alpha(M)}\int_0^\infty r^{\bar s_2} e^{2\pi |n-\alpha|r}\\ & \ \ \ \ \ \  \ \ \  \ \   \  \times\int_{\mathcal{P}}r_M^{s_1}\,e(i|m-\alpha|r_M-\langle m-\alpha,z_M\rangle) e^{2\pi i \langle n-\alpha, x+iy\rangle } \, dx\,dy \frac{dr}{r^3}\\ 
 & =  \sum_{\substack{M\in \Gamma _\infty'\backslash \Gamma\\c= 0} }\overline{\chi_\alpha\left(\begin{bmatrix} a& b \\ 0 & d \end{bmatrix}\right)}  \int_0^\infty r^{\bar s_2} e^{2\pi |n-\alpha|r}\\ &  \ \ \ \ \ \  \ \ \  \ \ \ \ \times\int_{\mathcal{P}}r^{s_1}  e^{-2\pi |m-\alpha|r -2\pi  i \langle \varphi_a^*(m-\alpha)-( n-\alpha) , x+iy\rangle } \, dx\,dy \frac{dr}{r^3}\\
  & = \text{vol}(\mathcal{P})C_{m,n}\int_0^\infty r^{s_1+\bar s_2-2}e^{-2\pi (|m-\alpha|-|n-\alpha|)r}\frac{dr}{r}\\ 
  & 
   =(2\pi )^{2-s_1-\bar s_2}(|m-\alpha|-|n-\alpha|)^{2-s_1-\bar s_2}\Gamma(s_1+\bar s_2-2)\text{vol}(\mathcal{P}) C_{m,n},
\end{align*}
where 
\[
C_{m,n} := \sum_{\substack{a,d\\ \varphi_a^*(m-\alpha)=n-\alpha}} \overline{\chi_\alpha\left(\begin{bmatrix} a& b \\ 0 & d \end{bmatrix}\right)}.
\] The $c=0$ case will be absorbed into $R(s)$.

We now turn to the contribution of the elements $M\in  \Gamma '_\infty\backslash \Gamma$ where $c\neq0$,
\begin{align*}
 \sum_{\substack{M\in \Gamma _\infty'\backslash \Gamma\\c\neq 0} }\overline{\chi_\alpha(M)} \int_0^\infty r^{\bar s_2} e^{2\pi |n-\alpha|r}\int_{\mathcal{P}}\,r_M^{s_1} e(i|m-\alpha|r_M&-\langle m-\alpha,z_M\rangle )\\ &\times  e^{2\pi i \langle n-\alpha, x+iy\rangle } \, dx\,dy \frac{dr}{r^3}.
 \end{align*}
 We first examine the inner integral and denote $H(r)$ by
 \[H(r):=
\sum_{\substack{M\in \Gamma _\infty'\backslash \Gamma\\c\neq 0} }\overline{\chi_\alpha(M)} \int_{\mathcal{P}}\,r_M^{s_1} e(i|m-\alpha|r_M-\langle m-\alpha,z_M\rangle )\, e^{2\pi i \langle n-\alpha, x+iy\rangle } \, dx\,dy.
 \]
 Since $\text{Im}(u\bar v) =- \text{Im}(\bar u v)$, we have
 \begin{align*}
H(r) =&\sum_{\substack{M\in \Gamma _\infty'\backslash \Gamma\\c\neq 0} }\overline{\chi_\alpha(M)}
 \int_{\mathcal{P}}\frac{r^{s_1}e^{(-2\pi |m-\alpha|r)/(|cz+ d|^2+|c|^2r^2)}}{(|cz+ d|^2+|c|^2r^2)^{s_1}}
 \\ &  \ \ \ \ \ \ \ \ \ \ \ \    \times \widetilde e\left(- \frac{\overline{m-\alpha}}{c}\left(\frac{\overline{cz+d}}{|cz+ d|^2+|c|^2r^2}\right)+\frac{a(\overline{m-\alpha})}{c} -(\overline{n-\alpha}) z\right)\, dx\,dy \\
\phantom{H(r) } & = \sum_{\substack{M\in \Gamma_\infty'\backslash \Gamma\\c\neq 0} }\overline{\chi_\alpha(M)}\,r^{s_1}\frac{\widetilde e\left(\frac{a(\overline{m-\alpha})}{c}\right)}{|c|^{2{s_1}}}
 \int_{\mathcal{P}}\frac{e^{(-2\pi |m-\alpha|r)/\left((|z+ d/c|^2+r^2)|c|^2\right)}}{(|z+ d/c|^2+r^2)^{s_1}}\\ 
 &  \ \ \ \ \ \ \ \ \ \ \ \ \ \ \ \ \ \ \ \ \ \ \ \ \ \ \ \ \ \ \ \ \ \   \times \widetilde e\left(- \frac{ \overline{m-\alpha}}{c^2}\left(\frac{\overline{z+d/c}}{|z+ d/c|^2+r^2}\right) -(\overline{n-\alpha}) z\right)\, dx\,dy\end{align*}
 \begin{align*}
 & = \sum_{\substack{(c,d)\in\mathcal{R}\\c\neq 0}}\overline{\chi_\alpha(M)}\,r^{s_1}\frac{\widetilde e\left(\frac{a( \overline{m-\alpha})}{c}\right)}{|c|^{2{s_1}}} \sum_{\ell \in\Lambda} \int_{\mathcal{P}}\frac{e^{(-2\pi |m-\alpha|r)/\left((|z+\ell+ d/c|^2+r^2)|c|^2\right)}}{(|z+\ell +d/c|^2+r^2)^{s_1}} \\
 &  \ \ \ \ \ \ \ \ \ \ \ \ \ \ \ \ \ \ \ \ \ \ \ \ \ \ \ \ \  \ \times \widetilde e\left( -\frac{\overline{ m-\alpha}}{c^2}\left(\frac{\overline{z+\ell +d/c}}{|z+\ell + d/c|^2+r^2}\right)- (\overline{n-\alpha}) z\right)\, dx\,dy\\
 &=\sum_{\substack{(c,d)\in\mathcal{R}\\c\neq 0}}\overline{\chi_\alpha(M)}\,r^{s_1}\frac{\widetilde e\left(\frac{a(\overline{m-\alpha})}{c}\right)}{|c|^{2{s_1}}} 
\int_{\mathbb{R}^2}\frac{e^{(-2\pi |m-\alpha|r)/\left((|z+ d/c|^2+r^2)|c|^2\right)}}{(|z+d/c|^2+r^2)^{s_1}}\\
 &  \ \ \ \ \ \ \ \ \ \ \ \ \ \ \ \ \ \ \ \ \ \ \ \ \ \ \ \ \ \ \ \ \ \    \times\widetilde e\left( -\frac{\overline{m-\alpha}}{c^2}\left(\frac{\overline{z+d/c}}{|z+ d/c|^2+r^2}\right) - (\overline{n-\alpha}) z\right)\, dx\,dy.
 \end{align*}
 Letting $z'= z+d/c$ and then $z=z'/r$ we have 

 \begin{align*}\label{eq:intexp2}
H(r) 
   & = \sum_{\substack{(c,d)\in\mathcal{R}\\c\neq 0}}\overline{\chi_\alpha(M)}\,r^{2-{s_1}}\frac{\widetilde e\left(\frac{a(\overline{m-\alpha})}{c}\right)}{|c|^{2{s_1}}}\int_{\mathbb{R}^2}\frac{e^{(-2\pi |m-\alpha|)/\left((|z|^2+1)r|c|^2\right)}}{(|z|^2+1)^{s_1}} \\ &  \ \ \ \ \ \ \ \ \ \ \ \ \ \ \ \ \ \ \ \ \ \ \ \ \ \ \ \ \ \ \ \ \ \ \times \widetilde e\left( -\left(\frac{(\overline{m-\alpha})\bar z}{c^2 r(|z|^2+1)}\right)-(\overline{n-\alpha}) (zr-d/c)\right)\, dx\,dy\nonumber\\
   & =\sum_{c\neq 0} S_\infty(m,n,c,\chi_\alpha)\frac{r^{2-{s_1}}}{|c|^{2{s_1}}} \int_{\mathbb{R}^2}\frac{e^{(-2\pi |m-\alpha|)/\left((|z|^2+1)r|c|^2\right)}}{(|z|^2+1)^{s_1}} \\&  \ \ \ \ \ \ \ \ \ \ \ \ \ \ \ \ \ \ \ \ \ \ \ \ \ \ \ \ \ \ \ \ \ \ \ \  \ \ \ \ \ \times  \widetilde e\left( -\left(\frac{(\overline{m-\alpha})\overline{z}}{c^2 r(|z|^2+1)}\right) -(\overline{n-\alpha}) zr\right)\, dx\,dy\nonumber,
\end{align*}
 where we recall that
$
 \displaystyle
S_\infty(m,n,c,\chi_\alpha)= \sum_{(c,d)\in\mathcal{R} }\overline{\chi_\alpha (\gamma)}\, \widetilde e\left(\frac{a(\overline{m-\alpha})+d(\overline{n-\alpha})}{c}\right).$

Finally we integrate
\begin{align*}\int_0^\infty  r^{\bar s_2} &e^{-2\pi |n-\alpha|r}H(r)\,  \frac{dr}{r^3} \\
 &= \sum_{c\neq 0}\frac{S_\infty(m,n,c,\chi_\alpha)}{|c|^{2{s_1}}}\int_0^\infty \int_{\mathbb{R}^2}r^{\bar s_2-s_1 + 2}e^{-2\pi |n-\alpha|r}\frac{e^{(-2\pi |m-\alpha|)/\left((|z|^2+1)r|c|^2\right)}}{(|z|^2+1)^{s_1}} \\&  \ \ \ \ \ \ \ \ \ \ \ \ \ \ \ \ \ \ \ \ \ \ \ \ \ \ \ \ \ \ \ \ \times  \widetilde e\left( -\left(\frac{(\overline{m-\alpha})\overline{z}}{c^2 r(|z|^2+1)}\right) -(\overline{n-\alpha}) zr\right)\, dx\,dy  \, \frac{dr}{r^3}.
\end{align*}
Letting $\bar s_2 = s_1+2$ and $s_1 =s$, then we have
 \begin{align*}\int_0^\infty  r^{\bar s_2} &e^{-2\pi |n-\alpha|r}H(r)\,  \frac{dr}{r^3} \\
&=\sum_{c\neq 0}\frac{S_{\infty}(m,n,c,\chi_\alpha)}{|c|^{2{s}}} \int_0^\infty \int_{\mathbb{R}^2}\frac{r^2 e^{-2\pi |n-\alpha| r} \widetilde e((\overline{n-\alpha})r z)}{(|z|^2+1)^s}\, dx\,dy  \, \frac{dr}{r}\\ 
&  \ \ \ \ \ \ \ \ \ \ \ \ \ \ \ \ \ \ \ \ \ \ \ \ \ \ \ \ \ \ \ \ \ \ \ \ \ \ \ \ \ \ \ \ \ \
  + \sum_{c\neq 0}\frac{S_{\infty}(m,n,c,\chi_\alpha)}{|c|^{2{s}}} R_{m,n}(s,c)
  \end{align*}
    \begin{align*} &=\mathcal{I}  \cdot \sum_{c\neq 0}\frac{S_{\infty}(m,n,c,\chi_\alpha)}{|c|^{2{s}}}
  + \sum_{c\neq 0}\frac{S_{\infty}(m,n,c,\chi_\alpha)}{|c|^{2{s}}} R_{m,n}(s,c),
\end{align*}
where $$\mathcal{I}:= \int_0^\infty \int_{\mathbb{R}^2}\frac{r^2 e^{-2\pi |n-\alpha| r} \widetilde e(-(\overline{n-\alpha})r z)}{(|z|^2+1)^s}\, dx\,dy  \, \frac{dr}{r}
$$
 and 
 \begin{align*}
 R_{m,n}(s,c) := \int_0^\infty \int_{\mathbb{R}^2}&\frac{r e^{-2\pi |n-\alpha| r} \widetilde e(-(\overline{n-\alpha})r z)}{(|z|^2+1)^s}\\
&\times \Big[\exp\left(\frac{-2\pi |m-\alpha|}{(|z|^2+1)r|c|^2}\right) \widetilde e\left( \frac{-(\overline{m-\alpha})\overline{z}}{c^2 r(|z|^2+1)} \right)-1\Big]\, dx\,dy  \, dr. \end{align*}

We evaluate $\mathcal{I}$, following  Proposition 3.10 of \cite{Sarnak}.
Examining the inner integral, for $n=n_1+in_2$ and $\alpha=a_1+ia_2$, we have
\begin{align*}
\int_{\mathbb{R}^2}\frac{ \widetilde e(-(\overline{n-\alpha})r z)}{(|z|^2+1)^s}\, dx\,dy&
= \int_{\mathbb{R}^2}\frac{ e^{-2\pi i r \left((n_1-a_1)y-(n_2-a_2)x\right)}}{(x^2+y^2+1)^s}\, dx\,dy\\
& = \int_0^\infty \int_0^{2\pi}\frac{-2\pi i r \rho\sin(\theta - \beta)|n-\alpha|}{(\rho^2+1)^s}\rho\,d\theta\,d\rho
\end{align*}
where $\beta$ is chosen so that $|n-\alpha|\cos\beta = n_1-a_1$ and $|n-\alpha|\sin\beta = n_2-a_2$. By a change of variables $\theta\mapsto\theta+\beta$, 
\begin{align*}
\int_{\mathbb{R}^2}\frac{ \widetilde e(-(\overline{n-\alpha})r z)}{(|z|^2+1)^s}\, dx\,dy&
 = \int_0^\infty\frac{\rho}{(\rho^2+1)^s} \int_0^{2\pi}e^{-2\pi i r \sin\theta |n-\alpha|}\,d\theta\,d\rho\\
 & = \int_0^\infty\frac{J_0(2\pi r |n-\alpha|\rho)\rho}{(\rho^2+1)^s}\,d\rho\\ & 
  = \frac{K_{1-s}(2\pi |n-\alpha|r)(2\pi |n-\alpha|r)^{s-1}}{\Gamma(s)2^{s-1}}.
\end{align*}
Thus 
\begin{align*}
\mathcal{I} &=\frac{1}{\Gamma(s)2^{s-1}} \int_0^\infty r^2 e^{-2\pi |n| r}K_{1-s}(2\pi |n-\alpha|r)(2\pi |n-\alpha|r)^{s-1} \frac{dr}{r}\\&
 =\frac{1}{\pi^{3/2} 4^{s+1} |n-\alpha|^2} \frac{\Gamma(2s)}{\Gamma(s)\Gamma(s+3/2)}
\end{align*}
as in \cite[p.272]{Sarnak}.

As in \cite{Sarnak} Proposition 3.10, $R_{m,n}(s,c)\ll \frac{|c|^{-2}}{\sigma - 1}$ for $\text{Re}(s)=\sigma$ and so $$\sum_{c\neq 0}\frac{S_{\infty}(m,n,c,\chi_\alpha)}{|c|^{2{s}}} R_{m,n}(s,c)$$ is holomorphic in $\text{Re}(s)>1$. 
\end{proof}

%
%

Finally, Stirling's formula implies that for $\text{Im}(s)=t$,
$$\left|\frac{\pi^{-3/2}4^{-s-1}|n-\alpha|^{-2}\Gamma(2s)}{\Gamma(s)\Gamma(s+3/2)} \right|\sim \frac{1}{4\pi ^2|n-\alpha|^2(2t+1)}$$
 as $|t|\to \infty$. From this and Lemmas \ref{lem1} and \ref{lem2}, we have Theorem \ref{thm:GSthm1}.

\begin{rek}[Congruence subgroups]

Following \cite[\S7.6]{EGM}, let $\fa$ be a nonzero ideal in $\mathcal O_K$, and let 
$
\Gamma(\fa) = \left\{A \in\PSL_2(\mathcal O_K) : A \equiv I \bmod \fa \right\}.
$ 
The lattice of translations corresponding to the unipotent part $\Gamma(\fa)'_\infty$ of the stabilizer of $\infty$ is $\fa\subset \mathbb C$. Consider the Poincar\'e series
\[
{U}_m(P,s, \chi,\Gamma,\fa)=\sum_{A\in \Gamma(\fa)'_\infty\backslash \Gamma(\fa)}  \overline{\chi(A)}r_A^s e(i|m|r_A+\langle m,z_A\rangle),
\]
where $\chi$ is an additive homomorphism on $\Gamma(\fa)$.  By \cite[Lemma 7.6.6]{EGM} its analytic behaviour is parallel to  that of $U_m(P,s, \chi_\alpha) $ above, and its inner product is computed explicitly in \cite[Proposition 7.6.11]{EGM} in the case $\chi = 1$.
\end{rek}

\subsection{Proof of Theorem \ref{thm:GSthm2}} 
 
We can now prove the desired bound on our Selberg-Kloosterman sums. Let $\epsilon >0$.  By the definition of $Z_\infty$ and the trivial bound on the Kloosterman sum, on $\text{Re}(s) =2+\epsilon$, we have $|Z_{\infty}(m,n,s,\chi_\alpha)| = O(1)$ while on $\text{Re}(s) =1+\epsilon$, $|Z_{\infty}(m,n,s,\chi_\alpha)| = O(|t|)$ where the implied constant depends on $\epsilon$. 
Recall that
 \begin{align*}
 Z_{\infty}(m,n,s,\chi_\alpha)&= \sum_{c\neq 0}\frac{S_{\infty}(m,n,c, \chi_\alpha)}{|c|^{2s}}.
 \end{align*}
By the Phragm\'{e}n-Lindel\"{o}f principle, there is an unique affine-linear function $g(\sigma)$ with $g(1+\epsilon) =1+\epsilon$ and $g(2+\epsilon)= 0$ so that 
\begin{equation}\label{eq:PL}
| Z_{\infty}(m,n,\sigma+it,\chi_\alpha)| = O\left(|t|^{g(\sigma)}\right),
\end{equation}
for all $1<\sigma\leq 2$ as $|t|\to \infty$.

Perron's formula gives
 \begin{align*}
\sum_{1\leq |c|\leq x}\frac{S_{\infty}(m,n,c,\chi_\alpha)}{|c|^2}
&= \frac{1}{2\pi i }\int_{(\sigma_\epsilon)} Z_{\infty}\left(m,n,1+\frac{s}{2},\chi_\alpha\right) \frac{x^{s}}{s}\,ds,
 \end{align*}
 where $\sigma_\epsilon :=2+\epsilon$,
and it follows that 
 \begin{align*}
\sum_{1\leq |c|\leq x}\frac{S_{\infty}(m,n,c,\chi_\alpha)}{|c|^2}
& =\frac{1}{2\pi i }\int_{\sigma_\epsilon-iT}^{\sigma_\epsilon+iT} Z_{\infty}\left(m,n,1+\frac{s}{2},\chi_\alpha\right) \frac{x^{s}}{s}\,ds + O\left(\frac{x^{\sigma_{\epsilon}}}{T}\right).
 \end{align*}
 Recall that $Z_{\infty}\left(m,n,1+\frac{s}{2},\chi_\alpha\right) $ has a finite number of simple poles at \\$s=2(s_j-1):=\beta_j$ with  $s_j\in(1,2)$  and define  $$\tau_j(m,n):=\frac{\text{Res}_{s=\beta_j}Z_{\infty}\left(m,n,1+\frac{s}{2},\chi_\alpha\right)}{{\beta_j} }. $$
We will now apply the Residue Theorem to the rectangular path of integration along the box $[1+iT, 1-iT]$, $[1-iT,\sigma_\epsilon-iT],\dots$ and from (\ref{eq:PL}) we have
 \begin{align*}
\sum_{1\leq |c|\leq x}\frac{S_{\infty}(m,n,c,\chi_\alpha)}{|c|^2}
& =\sum_{j} \tau_j(m,n)x^{\beta_j} + O\left(x^{\epsilon}T^{1\pm \epsilon}+\frac{x^{2+\epsilon}}{T}\right),
 \end{align*}
 as $x\to\infty$. Now choose $T=x$ and the result follows.

\section{Proof of equidistribution}

\subsection{Coset counting}

Define the counting function
\begin{equation}\label{eq:N}
N(X) :=\left| \left\{ \begin{pmatrix}*&*\\c&d\end{pmatrix} \in \mathcal R: |c| < X\right\}\right|,
\end{equation}
where we recall $\mathcal R$ is a complete set of double-coset representatives in \eqref{Rset}.
We derive an asymptotic for $N(X)$ using the Dirichlet series $\phi_0$ in the constant term of the Eisenstein series $E_A(P,s)$.

\begin{prop}
\label{counting}
For $N(X)$ as in \eqref{eq:N}, we have
\[
N(X) = \frac{|\Lambda|}{\mathrm{vol}(\Gamma)}X^4 + O\left(X^{3 + \epsilon}\right).
\]
\end{prop}

\begin{proof}
The proof follows the same method as that of Theorem \ref{thm:GSthm2} (see also \cite[Theorem 5]{B}), so we will be brief. In place of the series $Z_{\infty}(m,n,s,\chi_\alpha)$, we consider instead the Dirichlet series $\phi_0(s)$ arising from the constant term of the Eisenstein series $E_A(P,s)$, whose analytic properties we have recalled in Section \ref{eis}. 

We first write
\[
\phi_{0}(s) = \sum_{(c,d)\in\mathcal R}\frac{1}{|c|^{2s}} = \sum_{c\neq 0}\frac{1}{|c|^{2s}}\left|\left\{d: \begin{pmatrix}*&*\\c&d\end{pmatrix}\in \mathcal R\right\}\right|.
\]
(According to \cite{MNW}, the cardinality in the summand is at most that of $\{d: d \mod c\Lambda\}$, but we do not need this explicitly.)
By Perron's formula, we have again for $\sigma_\epsilon = 2 +\epsilon$,
\[
N(X) = \lim_{T\to\infty}\frac{1}{2\pi i}\int_{\sigma_\epsilon -iT}^{\sigma_\epsilon + iT}\phi_0(s)\frac{X^{2s}}{s}ds,
\]
and the same application of Stirling's formula and the Phragm\'en-Lindel\"of principle gives
\[
\frac{1}{2\pi i}\int_{\sigma_\epsilon -iT}^{\sigma_\epsilon + iT}\phi_0(s)\frac{X^{2s}}{s}ds + O\left(\frac{X^{4 + \epsilon}}{T}\right).
\]
Then shifting the line of integration to the line Re$(s) = 1$ and using the residue theorem, we have
\[
N(X) = \frac{|\Lambda|}{\text{vol}(\Gamma)}X^4 + O\left(X^2T +  \frac{X^{4 + \epsilon}}{T}\right)
\]
and taking $T=X$ yields the claim. 
\end{proof}

\subsection{Proof of Theorem 1}\label{proofothm1}

For the remainder of this paper we fix $\Gamma = \PSL_2(\mathcal O_K)$ and $\zeta= \infty$. Proposition \ref{counting} specializes to 
\begin{equation}
\label{NX}
N(X) = \frac{4\pi^2}{|d_K|^{3/2}\zeta_K(2)} X^4 + O\left(X^{3 + \epsilon}\right).
\end{equation}
The proof of the main theorem now follows, as explained in the introduction.
That is, returning to the setting of Theorem 1 and the discussion after it, the set under consideration can be rewritten as
\[
\{r\widetilde{D}(c,d)\,:\, (c,d)\in \mathcal{R}\},
\]
so the sum becomes
\[
\sum_{0 < |c| < X }\sum_{\substack{(c,d)\in \mathcal R}}e(rn\widetilde{D}(c,d)) = \sum_{(c,d)\in N(X)}e(rn\widetilde{D}(c,d)).
\]
The length of the sum is precisely $N(X)$, so combining Theorem 2 and \eqref{NX}, it follows by \eqref{def:S1} that the latter equals
\[ 
\sum_{|c|<X} S_\infty(\lfloor -r\rfloor,\lfloor-r\rfloor,c,{\chi_\alpha})  = o(X^4),
\]
as $X\to\infty$, and by the Weyl criterion we conclude the main result.

\subsection*{Acknowledgments} This work was begun through the Rethinking Number Theory 2 Workshop in 2021. K-L acknowledges support from NSF grant number DMS-2001909. W. was partially supported by DMS-2212924. We would also like to thank the referees for their thorough and thoughtful comments.

Data sharing not applicable to this article as no datasets were generated or analysed during the current study.

The authors have no relevant financial or non-financial interests to disclose.
 

\end{document}